\documentclass[a4paper]{article}

\usepackage[mathscr]{eucal}
\usepackage{amssymb}
\usepackage{latexsym}
\usepackage{amsthm}
\usepackage{amsmath}
\usepackage[dvips]{graphicx}
\usepackage{psfrag}

\newtheorem{theorem}{Theorem}[section]

\newtheorem{lemma}[theorem]{Lemma}
\newtheorem{corollary}[theorem]{Corollary}
\theoremstyle{definition}

\newtheorem{remark}[theorem]{Remark}

\newtheorem{example}[theorem]{Example}

\makeatletter
\@addtoreset{equation}{section}

\makeatother

\DeclareMathOperator{\rank}{rank}

\title{Bounds on $s$-distance sets with strength $t$} 
\author{Hiroshi Nozaki, Sho Suda}

\begin{document}
\maketitle

\renewcommand{\thefootnote}{\fnsymbol{footnote}}
\footnote[0]{Both of authors are supported by JSPS Research Fellowship.}
\footnote[0] {The first author stays at the University of Texas at Brownsville from August 24th, 2009 to August 23rd, 2010.}

\begin{abstract}
A finite set $X$ in the Euclidean unit sphere is called an $s$-distance set if the set of distances 
between any distinct two elements of $X$ has size $s$. 
We say that $t$ is the strength of $X$ if $X$ is a spherical $t$-design but not a spherical $(t+1)$-design. 
Delsarte-Goethals-Seidel gave an absolute bound for the cardinality of an $s$-distance set. 
The results of Neumaier and Cameron-Goethals-Seidel imply that if $X$ is a spherical $2$-distance set with strength $2$, then the known absolute bound for $2$-distance sets is improved. 
This bound are also regarded as that for a strongly regular graph with the certain condition of the Krein parameters. 
 
In this paper, we give two generalizations of this bound to spherical $s$-distance sets with strength $t$ (more generally, to $s$-distance sets with strength $t$ in a two-point-homogeneous space), and 
to $Q$-polynomial association schemes. 
 First,  for any $s$ and $s-1 \leq t \leq 2s-2$, we improve the known absolute bound for the size of a spherical $s$-distance set with strength $t$. 
Secondly, for any $d$, we give an absolute bound for the size of a $Q$-polynomial association scheme of class $d$ with the certain conditions of the Krein parameters.

\end{abstract}


\section{Introduction} 
	Delsarte, Goethals and Seidel \cite{Delsarte-Goethals-Seidel} introduced the concept of spherical $t$-designs. 
	Let ${\rm Harm}_i(\mathbb{R}^m)$ be the linear space of all real homogeneous harmonic polynomials of degree $i$, with $m$ variables. 
	A finite subset $X$ in the Euclidean unit sphere $S^{m-1}$ is called a spherical $t$-design 
if $\sum_{x\in X} \varphi(x)=0$ for any $\varphi \in {\rm Harm}_i(\mathbb{R}^m)$ and any $1 \leq i \leq t$.
	We say that $t$ is the strength of $X$ if $X$ is a spherical $t$-design but not a spherical $(t+1)$-design.
	We have an absolute lower bound for the cardinality of a spherical $2l$-design in $S^{m-1}$, namely 
	\[
	|X| \geq h_0+h_1+ \cdots +h_{l},
	\]
	where $|\ast|$ denotes the cardinality, and $h_i=h_{i,m}=\dim {\rm Harm}_i(\mathbb{R}^m)$. 
	Indeed, $h_{i,m}=\binom{m+i-1}{i}-\binom{m+i-3}{i-2}$, $h_{0,m}=1$, and $h_{1,m}=m$.
	
	The size of the set of distances between distinct two elements of $X\subset S^{m-1}$ is one of important parameters of spherical codes. 
	A finite $X \subset S^{m-1}$ is called an $s$-distance set if $|A(X)| =s$, where $A(X):=\{(x,y) \mid x,y \in X, x\ne y\}$ and $(,)$ denotes the usual inner product. 
	An absolute upper bound for an $s$-distance set is  
	\[
	|X| \leq h_0 +h_1 + \cdots +h_s. 
	\]
	 
	We have the inequality $t \leq 2s$ for a spherical $s$-distance set with strength $t$ \cite{Delsarte-Goethals-Seidel}. 
	One of fundamental results related with the theory of association schemes is that 
if $t \geq 2s-2$, then $X$ carries a $Q$-polynomial association scheme of class $s$ \cite{Delsarte-Goethals-Seidel}.  

	A spherical $2$-distance set with strength at least $2$ carries an association scheme of class $2$, that is, a strongly regular graph. 
	Conversely, a strongly regular graph is embedded to the unit sphere as a $2$-distance set with strength at least $2$ faithfully ({\it i.e.} the map is injective) \cite{Cameron-Goethals-Seidel}. 
	
	Cameron, Goethals and Seidel \cite{Cameron-Goethals-Seidel} proved that if the Krein parameter 
$q_{1,1}^1$ of a strongly regular graph is not equal to $0$, then the spherical embedding with respect to 
the primitive idempotent $E_1$ has strength $2$ as a design in $S^{m_1-1}$, where $m_1$ is the 
rank of $E_1$. 
	
	On the other hand, Neumaier \cite{Neumaier} proved that if $q_{1,1}^1 \ne 0$, then the size of the vertex set of the strongly regular graph is bounded above by $m_1(m_1+1)/2=h_{0,m_1}+h_{2,m_1}$. 
	
	These two results, due to Cameron-Goethals-Seidel and Neumaier, imply that if $X \subset S^{m-1}$ is a 
spherical $2$-distance set with strength $2$, then $|X|\leq h_{0,m}+h_{2,m}$. 
Remark that the known absolute bound for a $2$-distance set is improved because of the assumption of $t$-designs. 

We have several examples attaining the bound $|X|\leq h_{0,m_1}+h_{2,m_1}$, namely triangular graphs 
\cite{Hoffman, Connor, Shrikhande}, Chang graphs \cite{Chang}, and the strongly regular graphs obtained from the 
regular two-graph on $276$ vertices \cite{Goethals-Seidel,Haemers-Tonchev}.  
 
In this paper, we give two generalizations of this bound to $s$-distance sets with strength $t$ in $S^{m-1}$ (more generally, a 
two-point-homogeneous space), and to $Q$-polynomial association schemes of class $s$. We prove the following bounds. 
 
Let $X$ be a spherical $s$-distance set with strength $2s-i$ where $2 \leq i \leq s+1$. Then, 
\begin{equation} \label{design_intro}
|X| \leq \sum_{k=0}^s h_k-h_{s-i+1}. 
\end{equation}
When $s=2$ and $i=2$, this bound coincides with $|X| \leq h_0+h_2$. 

Let $(X,\mathcal{R})$ be an $s$-class $Q$-polynomial scheme with respect to the ordering $E_0,E_1,\ldots,E_s$.
Define $l=\max\{k\in\{0,1,\ldots,s\}\mid q_{1,0}^0=q_{1,1}^1=\cdots=q_{1,k}^k=0\}$.
If $(s-1)/2\leq l\leq s-1$ holds, then 
\begin{equation} \label{ass_intro_1} 
|X|\leq \sum_{i=0}^{s} h_{i,m_1} - \underset{i \equiv s-1 \pmod 2}{\sum_{i = 2l-s+3}^{s-1}} h_{i,m_1}.
\end{equation}
If $l\leq (s-2)/2$ holds, then  
\begin{equation} \label{ass_intro_2}
|X|\leq \underset{i\equiv s \pmod 2}{\sum_{i=0}^s} h_{i,m_1}.
\end{equation}
When $s=2$ and $l=0$, this bound coincides with $|X| \leq h_{0,m_1}+h_{2,m_1}$. 
 
We have an example attaining the above bounds. 
There exists a finite subset $X$ of the minimum vectors of the Leech lattice in $\mathbb{R}^{24}$, such that 
$|X|=2025$, $X$ is in a $22$-dimensional affine subspace, and $X$ is a spherical $3$-distance set with strength $4$
 in $S^{21}$ after rescaling the norm to $1$. 
Then $|X|=h_0+h_1+h_3$, and hence $X$ attains the bound (\ref{design_intro}). 
The finite set $X$ carries a $Q$-polynomial association scheme of class $3$ because $t=2s-2$ holds. 
Then, $l=1$ and $X$ attains the bound (\ref{ass_intro_1}).  
 
We also prove a similar upper bound for antipodal spherical $s$-distance sets with strength $t$. 
The vertex set of the dodecahedron is an antipodal spherical $5$-distance set with strength $5$ in $S^2$, and attains the upper bound.    

\section{Preliminaries}
	
	\subsection{Two-point-homogeneous spaces}
	In this section, we introduce the concept of two-point-homogeneous spaces \cite[Chapter 9]{Conway-Slaone}, \cite{Kabatyansky-Levenshtein}. 
	
	Let $M$ be a compact metric space with a distance $\rho$ on it, and $\tau$ is the certain 
	function in $\rho$ ({\it i.e.} $\tau(x,y)=F(\rho(x,y))$). 
	We call $M$ a two-point-homogeneous space if there exists a group $G$ acting on $M$ such that the following assumption hold: For any $x,x',y,y' \in M$, we have $\tau(x,y)=\tau(x',y')$ if and only if there is an element $g \in G$ such that $g(x)=x'$ and $g(y)=y'$.  
	
	Let $\mu$ be the Haar measure, which is invariant under $G$. We assume that $\mu$ is normalized so that $\mu(M)=1$.
	Let $L^2(M)$ denote the vector space of complex-valued functions $u$ on $M$, satisfying 
	\[
	\int_M |u(x)|^2 d\mu(x) < \infty
	\]
	with inner product 
	\[
	(u_1,u_2)=\int_M u_1(x) \overline{u_2(x)} d\mu(x). 
	\]
	The space $L^2(M)$ decomposes into a countable direct sum of mutually orthogonal subspaces $\{ V_k \}_{k=0,1,\ldots}$ called (generalized) spherical harmonics. 
	Let $\{ \phi_{k,i} \}_{i=1}^{h_k}$ be an orthonormal basis for $V_k$, where $h_k= \dim V_k$. 
	Since $M$ is distance transitive, the function 
	\[
	\Phi_k(x,y):=\sum_{i=1}^{h_k} \phi_{k,i}(x) \overline{\phi_{k,i}(y)}
	\]
	depends only on $\tau(x,y)$. 
	This expression is called the addition formula, and $\Phi_k(\tau)$ is called the zonal spherical function associated with $V_k$.  It is immediate from the definition that $\Phi_k$ is positive definite, that is, 
	\[
	\sum_{x \in X} \sum_{y\in X}\Phi_k(\tau(x,y)) \geq 0
	\]
	for any $X \subset M$.
	Throughout this paper, we assume that $\{\Phi_i \}$ forms a family of orthogonal polynomials. 
Remark that for all known two-point-homogeneous spaces, $\Phi_i$ are polynomials. 
 We suppose that the degree of $\Phi_k$ is $k$. Note that $\Phi_k(\tau_0)=h_k$, where $\tau_0=\tau(x,x)$ for $x \in X$.

	\begin{example}
	The unit sphere $S^{m-1} \subset \mathbb{R}^m$ is a two-point-homogeneous space. 
	A polynomial with $m$ variables is said to be harmonic if its Laplacian is equal to zero. 
	Then, $V_k$ is the linear space of all homogeneous harmonic polynomials of degree $k$,
 with $m$ variables, and $h_k=\binom{m+k-1}{k} - \binom{m+k-3}{k-2}$. 
	The polynomial $\Phi_k$ is the Gegenbauer polynomial $G_k(x)$. 
	The Gegenbauer polynomials $G_k$ are defined by the following manner: 
	\[
	x G_k(x)=\lambda_{k+1} G_{k+1}(x)+(1-\lambda_{k-1})G_{k-1}(x)
	\]
	where $\lambda_k=k/(m+2k-2)$, $G_0(x) \equiv 1$, and $G_1(x)= mx$. 
	\end{example}
	
\subsection{Distance sets and $t$-designs}
	Let $M$ be a two-point-homogeneous space.
	We define 
	\[
	A(X)=\{\tau(x,y) \mid x,y \in X, x\ne y \}
	\]
	for a finite set $X$ in $M$. 
	A finite $X \subset M$ is called an $s$-distance set (or $s$-code) if $|A(X)|=s$. 
	If $X$ is an $s$-distnace set, then $|X|\leq \sum_{i=0}^s h_i$ \cite{Delsarte-Goethals-Seidel}.

	A finite $X\subset M$ is called a $t$-design if 
	\[
	\sum_{x,y \in X} \Phi_i(\tau(x,y))=0
	\]
	for any $1\leq i \leq t$. 
	We say that $X$ has strength $t$ if $X$ is a $t$-design but not a $(t+1)$-design. 
	It is proved by the same method in \cite{Delsarte-Goethals-Seidel} that if $X$ is a 2e-design, then $|X| \geq \sum_{i=0}^e h_i$.
	
	Let $\{\phi_{k,i} \}_{i=1}^{h_k}$ be an orthonormal basis of $V_k$. 
	Let $H_k$ be the matrix whose $(i,j)$-entry is $\phi_{k,j}(x_i)$, where $X=\{x_1,x_2,\ldots, x_n\} \subset M$. The matrix $H_k$ is called the characteristic matrix of 
	degree $k$.

	The following are needed later. 
	\begin{theorem}[\cite{Koornwinder}] \label{koon}
	The product of any two zonal spherical functions $\Phi_i(t)$ and $\Phi_j(t)$ can be expressed as 
	\[
	\Phi_i(t) \Phi_j(t)=\sum_{k=0}^{i+j} c_{i,j}^k \Phi_k(t)
	\]
	with $c_{i,j}^k \geq 0$ and $c_{i,j}^0=h_i \delta_{i,j}$. 
	\end{theorem}
	\begin{lemma} \label{f_0}
	Let $F(t)= \sum_{k}f_k \Phi_k(t)$, and $G(t)=\Phi_l(t) F(t)/ h_l=\sum_{k}g_k \Phi_k(t)$. 
	Then $g_0=f_l$. 
	\end{lemma}
	\begin{proof}
	By Theorem \ref{koon}, 
	\[
	\sum_{k}g_k \Phi_k(t)=\frac{\Phi_l(t)}{h_l} \sum_{k}f_k \Phi_k(t) =\frac{1}{h_l} \sum_{k}f_k 
	\sum_{i=0}^{k+l} c_{l,k}^i \Phi_i(t).
	\]
	Since $c_{i,j}^0=h_i \delta_{i,j}$, this lemma follows. 
	\end{proof}
	Define $||N||^2 = \sum_{i,j} n_{i,j}^2$, where $n_{i,j}$ is the $(i,j)$-entry of a matrix $N$.
	Let $^tN$ be the transpose of a matrix $N$. 
	\begin{theorem}[\cite{Delsarte-Goethals-Seidel}] \label{ani}
	Let $X \subset M$ and $F(t)= \sum_{k=0}^{\infty} f_k \Phi_k(t)$. If $F(\alpha)=0$ for any $\alpha \in
	A(X)$ and $F(\tau_0)=1$, then 
	\[
	|X|(1-|X|f_0)= \sum_{k=1}^{\infty} f_k ||{^tH_k}H_0||^2. 
	\]
	\end{theorem}
	\begin{theorem}[\cite{Delsarte-Goethals-Seidel}] \label{design}
	If $X$ is a $t$-design in $M$,  
	then  for nonnegative integers $k,l$ such that $k+l\leq t$,
	\[ ^tH_kH_l =|X| \Delta_{k,l}, \] 
	where $\Delta_{k,l}$ is the identity matrix if $k=l$, and $\Delta_{k,l}$ is the zero matrix if $k\ne l$. 
	\end{theorem}
	We define \[F_X(t):=\prod_{\alpha \in A(X)} \frac{t- \alpha}{\tau_0-\alpha}. \] 
\begin{lemma} \label{1/|X|}
Let $X$ be an $s$-distance set with strength $t$ in $M$. 
We have $F_X(t)=\sum_{k=0}^s f_k \Phi_k(t)$. 
If $t \geq s-1$, then $f_{t-s+1}\ne 1/|X|$. 
\end{lemma}
\begin{proof}
Let $G(t)=\Phi_{t-s+1}(t) F_X(t)/h_{t-s+1}=\sum_{k=0}^{t+1} g_k \Phi_k(t)$. 
By Theorem \ref{ani}, 
	\[
	|X|(1-|X|g_0)= \sum_{k=1}^{t+1} g_k ||{^tH_k}H_0||^2, 
	\]
	where $g_{t+1} \ne 0$. 
By Theorem \ref{design}, $|X|(1-|X|g_0)= g_{t+1} ||{^tH_{t+1}}H_0||^2$. 
Since $X$ is not a $(t+1)$-design, $||{^tH_{t+1}}H_0||^2 \ne 0$ and hence $1-|X|g_0\ne 0$. 
By Lemma \ref{f_0}, $f_{t-s+1}=g_0$ and $f_{t-s+1}\ne 1/|X|$.
\end{proof}
	

	\subsection{Cometric association schemes} 
	Let $X$ be a finite set and $\mathcal{R}=\{R_0,R_1,\ldots,R_s\}$ be a set of non-empty subsets of $X\times X$.
For $0\leq i\leq s$, let $A_i$ be the $(0,1)$-matrix indexed by the elements of $X$, whose $(x,y)$-entry is $1$ if $(x,y)\in R_i$, and $0$ otherwise.
The matrix $A_i$ is called the adjacency matrix of the graph $(X,R_i)$.
A pair $(X,\mathcal{R})$ is a symmetric association scheme of class $s$ if the following hold:
\begin{enumerate}
\item $A_0 $ is the identity matrix;
\item $\sum_{i=0}^sA_i=J$, where $J$ is the all one matrix;
\item ${}^t A_i=A_i$ for $1\leq i\leq s$;
\item $A_iA_j$ is a linear combination of $A_0,A_1,\ldots,A_s$ for $0\leq i,j\leq s$.
\end{enumerate}
The vector space $\mathcal{A}$ spanned by $A_i$ over the real field $\mathbb{R}$ is an algebra which is called the Bose-Mesner algebra of $(X,\mathcal{R})$.
Since $\mathcal{A}$ is semi-simple and commutative, there exist primitive, mutually orthogonal idempotents $\{E_0,E_1,\ldots,E_s\}$ where $E_0=\frac{1}{|X|}J$.
Let $m_i$ be the rank of $E_i$ for $0\leq i\leq s$.
Since $\sum_{i=0}^sE_i=I$ and $\{E_0,E_1,\ldots,E_s\}$ are mutually orthogonal idempotents, 
\begin{align}\label{S5}
\sum_{i=0}^sm_i=|X|.
\end{align}
Since $\mathcal{A}$ is closed under the entry-wise product $\circ$, we define the Krein parameters $q_{i,j}^k$ by 
\begin{equation} E_i\circ E_j=\frac{1}{|X|}\sum_{k=0}^s q_{i,j}^k E_k \label{S:eq}. \end{equation}
The Krein parameters are nonnegative real numbers \cite[Theorem~3.8]{BI} \cite{Scott}. 
The scheme $(X,\mathcal{R})$ is $Q$-polynomial (or cometric) with respect to the ordering $E_0,$ $E_1,\ldots,E_s$ if the following hold:
$q_{1,j}^k>0$ if $k=j\pm1$ and $q_{1,j}^k=0$ if $k<j-1$ or $j+1<k$.
If $(X,\mathcal{R})$ is $Q$-polynomial, for notational convenience, set $a_i^*=q_{1,i}^i$ $(0\leq i\leq s)$, $b_i^*=q_{1,i+1}^i (0\leq i\leq s-1)$, and $c_i^*=q_{1,i-1}^i, (1\leq i\leq s)$, $c_0^*=b_s^*=0$, 
we abbreviate $m_1$ as $m$.
The matrix $B_1^*=(q_{1,j}^k)_{0\leq j,k \leq s}$ is said to be the Krein matrix.
It follows from \cite[Proposition~3.7]{BI} that $a_0^*=0$, $c_1^*=1$ and
\begin{align}
& a_i^*+b_i^*+c_i^*=m \text{ for } 0\leq i\leq s, \label{S4}\\
& b_i^*m_i=c_{i+1}^*m_{i+1} \text{ for } 0\leq i\leq s-1. \label{S5}
\end{align}
Then (\ref{S:eq}) gives the following three-term recurrence:
$E_1\circ E_i=\frac{1}{|X|}(c_{i+1}^*E_{i+1}+a_i^*E_i+b_{i-1}^*E_{i-1}) \text{ for }1\leq i\leq s-1$. 
Denote $A^{\circ i}$ by the product $A\circ A\circ \cdots \circ A$ to $i$ factors.
The following lemma is used in a proof of Theorem~\ref{S0}:
\begin{lemma}[{\cite[Lemma~4.7]{BI}}] \label{S3}
	Let $A$ be a square matrix of rank $m$ over real field $\mathbb{R}$. 
	The following hold:
	\begin{enumerate}
    \item $\rank A^{\circ h}\leq \binom{m+h-1}{h}$ for any nonnegative integer $h$.
    \item If the equality in $(1)$ holds for some $h$, then the equality in $(1)$ holds for any $j\leq h$.
    \end{enumerate}
\begin{proof}
(1): 
Let $\{a_1,\ldots,a_m\}$ be a basis of the vector space spanned by the rows of $A$.
Since the row space of $A^{\circ h}$ are spanned by a set $\{a_{i_1}\circ\cdots\circ a_{i_h}\mid 1\leq i_1\leq \ldots\leq i_h\leq m\}$, 
the desired result follows.\\
(2): 
Suppose that $\rank {A^{\circ h}}=\binom{m+h-1}{h}$ for some nonnegative integer $h$.
This is equivalent that the set $\{a_{i_1}\circ\cdots\circ a_{i_h}\mid 1\leq i_1\leq \ldots\leq i_h\leq m\}$ is linearly independent.
Then a set $\{a_{i_1}\circ\cdots\circ a_{i_j}\mid 1\leq i_1\leq \ldots\leq i_j\leq m\}$ is also linearly independent for any $j\leq h$.
Indeed, for $j\leq h$, assume $\sum c_{i_1,\ldots,i_j}a_{i_1}\circ\cdots\circ a_{i_j}=0$ for some $c_{i_1,\ldots,i_j}\in\mathbb{R}$, where indices run through $1\leq i_1\leq \ldots\leq i_j\leq m$.
Multiplying $a_1^{\circ(h-j)}$, we have $\sum c_{i_1,\ldots,i_j}a_1^{\circ (h-j)}\circ a_{i_1}\circ\cdots\circ a_{i_j}=0$.
Since the set $\{a_{i_1}\circ\cdots\circ a_{i_h}\mid 1\leq i_1\leq \ldots\leq i_h\leq m\}$ is linear independent, $c_{i_1,\ldots,i_j}=0$ for all indices.
Therefore the desired result is proved.
\end{proof}
\end{lemma}

We can consider the embedding of a $Q$-polynomial association scheme into the unit sphere as follows.
Since the primitive idempotent $E_1$ is positive semi-definite, 
there exists a $|X|$ times $m$ matrix $U$ of rank $m$ such that $\frac{|X|}{m}E_1=U{}^tU$.
Since $E_1$ has no repeated rows, $U$ has also no repeated rows.
Corresponding $x$ in $X$ to the $x$-th row vector of $U$, we identify $X$ as the row vectors $U$.
Then $X$ is always a spherical $2$-design.
$X$ is a spherical $3$-design if and only if $a_1^*=q_{1,1}^1=0$.
Further, we have a characterization for $X$ to be a spherical $t$-design in terms of the Krein parameters as follows:
\begin{theorem}[{\cite[Theorem~3.1]{S}}] \label{SHO}
Let $(X,\mathcal{R})$ be a $Q$-polynomial scheme.
Then the following are equivalent:
\begin{enumerate}
\item $X$ is a spherical $t$-design.
\item $a_i^*=0$ for $0\leq i\leq \lfloor (t-1)/2\rfloor$ and $c_j^*=\frac{mj}{m+2j-2}$ for $0\leq j\leq \lceil (t-1)/2\rceil$.
\end{enumerate}
\end{theorem}
	
\section{Bounds on $s$-distance sets with strength $t$}  
Let $D_i=H_i {^tH_i}$, and $\mathcal{E}_+^{(i)}$ denote the direct sum of the eigenspaces 
corresponding to the all positive eigenvalues of $D_i$. 
\begin{lemma} \label{dim}
The inequality $\dim \mathcal{E}_+^{(i)} \leq h_i$ holds. 
\end{lemma} 
\begin{proof}
Since $D_i$ is positive semidefinite, the rank of $D_i$ is equal to $\dim \mathcal{E}_+^{(i)}$. 
Note that the rank of $H_i$ is at most $h_i$. Therefore the rank of $D_i$ is at most $h_i$. 
\end{proof}
Let $\mathcal{E}_0^{(i)}$ denote the eigenspace corresponding to the zero eigenvalue of $D_i$. 
For each $x \in X$, let $e_x$ be the column vector, whose $x$-th entry is $1$, and other entries are $0$. 
Let $V$ denote the real vector space spanned by $\{e_x\}_{x \in X}$.  
Since $D_i$ is a positive semidefinite matrix, 
$V=\mathcal{E}_+^{(i)} \oplus \mathcal{E}_0^{(i)}$ for each $i$.
\begin{lemma} \label{neg}
Let $X$ be an $s$-distance set in $M$. 
Suppose we have $F_X(t)=\sum_{k=0}^s f_k \Phi_k(t)$, where $f_k$ are real numbers. 
Then 
\[
V= \sum_{k: f_k>0} \mathcal{E}_+^{(k)}.
\]
\end{lemma}
\begin{proof}
Suppose there exists $v \not\in \sum_{k: f_k>0} \mathcal{E}_+^{(k)}$. 
Then, we can write $v=v_1+v_2$, where $v_1 \in \sum_{k: f_k>0} \mathcal{E}_+^{(k)}$ and 
$0 \ne v_2 \in \bigcap_{k: f_k>0}\mathcal{E}_0^{(k)}$. 
Note that $I=\sum_{k=0}^s f_k D_k$. 
Therefore, 
\begin{align*}
v_2&= \sum_{k=0}^s f_k D_k v_2 \\
&= \sum_{k: f_k < 0} f_k D_k v_2. 
\end{align*}
Then, the matrix $\sum_{k: f_k < 0} f_k D_k$ has an eigenvalue $1$. 
This contradicts that $\sum_{k: f_k < 0} f_k D_k$ is negative semidefinite.  
\end{proof}

\begin{lemma} \label{reduce}
Let $X$ be an $s$-distance set with strength $2s-i$ in $M$, where $2 \leq i \leq 2s$.  
Suppose $F_X(t)=\sum_{k=0}^s f_k \Phi_k(t)$, where $f_i$ are real numbers, and $f_{j} \ne 1/|X|$ for some $\max\{s-i+1,0\} \leq j \leq \lfloor (2s-i)/2 \rfloor$. 
Then, we have $ \mathcal{E}_+^{(j)} \subset W$, where 
\[
 W=\sum_{k=\lfloor s-\frac{i}{2} \rfloor+1}^s \mathcal{E}_+^{(k)}. 
\] 
\end{lemma}

\begin{proof}
By Theorem \ref{design}, 
\begin{align*}
D_{j}&=\sum_{k=0}^s f_k D_k D_{j} \\
&=f_{j} |X| D_{j}+ \sum_{k=\lfloor s-\frac{i}{2} \rfloor+1}^s f_k D_k D_{j}, 
\end{align*}
for $\max\{s-i+1,0\} \leq j \leq \lfloor (2s-i)/2 \rfloor$. 
Therefore,
\[
(1-f_{j} |X| )D_{j}= \sum_{k=\lfloor s-\frac{i}{2} \rfloor+1}^s f_k D_k D_{j}. 
\]
where $1-f_{j} |X|\ne 0$. 
Let $v$ be an eigenvector corresponding to an eigenvalue $\lambda>0$ of $D_j$. 
We can write $v=\sum_{m}  v_m^{(k)}$ for each $k$, where $v_m^{(k)}$ is an eigenvector corresponding to an eigenvalue $\lambda_m^{(k)}$ of $D_k$.  
Then, 
\begin{align} 
(1-f_{j} |X| )D_{j}v&= \sum_{k=\lfloor s-\frac{i}{2} \rfloor+1}^s f_k D_k D_{j}v \nonumber \\
(1-f_{j} |X| )\lambda v&= \lambda \sum_{k=\lfloor s-\frac{i}{2} \rfloor+1}^s f_k D_k v \nonumber \\
&= \lambda \sum_{k=\lfloor s-\frac{i}{2} \rfloor+1}^s f_k D_k \sum_{m}  v_m^{(k)} \nonumber \\
&= \lambda \sum_{k=\lfloor s-\frac{i}{2} \rfloor+1}^s f_k  \sum_{m} \lambda_m^{(k)}  v_m^{(k)}. \label{last} 
\end{align}
Note that eigenvectors corresponding to the zero eigenvalue vanishes in (\ref{last}). 
Hence we have $v\in W$. 
The set of eigenvectors corresponding to positive eigenvalues of $D_j$ is a basis of $\mathcal{E}_+^{(j)}$. 
Therefore, $\mathcal{E}_+^{(j)} \subset W$. 
\end{proof}
\begin{remark} \label{rem}
Suppose the conditions in Lemma \ref{reduce}. When $s-i\geq 0$, we have $f_j=1/|X|$ for $0 \leq j \leq s-i$ \cite{Delsarte-Goethals-Seidel}.
\end{remark} 
The following is the main theorem in this section. 
\begin{theorem} \label{main}
Let $X$ be an $s$-distance set with strength $2s-i$ in $M$, where $2 \leq i \leq 2s$.  Suppose $F_X(t)=\sum_{k=0}^s f_k \Phi_k(t)$, where $f_k$ are real numbers. Then, 
\[
|X| \leq 
\sum_{k=0}^{s-i} h_k 
+ \underset{k: f_k = \frac{1}{|X|}}{\sum_{k=\max\{s-i+1,0\}}^{\lfloor s- \frac{i}{2} \rfloor}} h_k 
+\sum_{k=\lfloor s- \frac{i}{2} \rfloor +1}^{s} h_k, 
\]
where $\sum_{k=0}^{s-i} h_k=0$ if $s-i<0$. 
\end{theorem}

\begin{proof}
By Remark \ref{rem} and Lemmas \ref{dim}, \ref{neg}, and \ref{reduce}, 
\begin{align*}
|X| &= \dim (\sum_{k: f_k>0} \mathcal{E}_+^{(k)}) \\
&\leq \dim (\sum_{k=0}^s \mathcal{E}_+^{(k)}) \\
&= \dim(\sum_{k=0}^{s-i} \mathcal{E}_+^{(k)} \oplus \underset{k: f_k = \frac{1}{|X|}}{\sum_{k=\max\{s-i+1,0\}}^{
\lfloor s- \frac{i}{2} \rfloor}} \mathcal{E}_+^{(k)} \oplus \sum_{k=\lfloor s- 
\frac{i}{2} \rfloor+1}^{s} \mathcal{E}_+^{(k)})\\
&\leq \sum_{k=0}^{s-i} h_k + \underset{k: f_k = \frac{1}{|X|}}{\sum_{k=\max\{s-i+1,0\}}^{
\lfloor s- \frac{i}{2} \rfloor}} h_k +\sum_{k=\lfloor s- \frac{i}{2} \rfloor+1}^{s} h_k. 
\end{align*}
\end{proof}
\begin{corollary}\label{corollary}
Let $X$ be an $s$-distance set with strength $2s-i$ in $M$,  
where $2 \leq i \leq s+1$. Then, 
\[
|X| \leq \sum_{k=0}^s h_k-h_{s-i+1}. 
\]
\end{corollary}
\begin{proof}
By Lemma \ref{1/|X|}, $f_{s-i+1}\ne 1/|X|$. By Theorem \ref{main}, this corollary follows. 
\end{proof}
\begin{remark}
If we prove $f_k \ne 1/|X|$ for some $s-i+2 \leq k \leq \lfloor s- \frac{i}{2} \rfloor$ under the assumption in Corollary \ref{corollary}, 
then the upper bound is improved. 
\end{remark}
A finite $X \subset S^{m-1}$ is said to be antipodal if $-x \in X$ for any $x\in X$. 
Let $\delta_s=1$ if $s$ is odd, and $\delta_s=0$ if $s$ is even. 
\begin{corollary} \label{anti_coro}
Let $X$ be an antipodal $s$-distance set with strength $2s-2i-1$ in $S^{m-1}$, 
where $1+\delta_s \leq i \leq s + \delta_s$. Then, 
\[
|X| \leq 2 \sum_{k=0}^{\frac{s-\delta_s}{2}} h_{2k}-2 h_{s+\delta_s-2i}. 
\]
\end{corollary}
\begin{proof}
The finite set $X$ is identified with a $|X|/2$ point $((s-\delta_s)/2)$-distance set with strength $s-i-1$ in the real projective space \cite[Theorem 9.2]{Levenshtein}. Note that the real projective space is a two-point-homogeneous space. 
Let $h_k$ be the dimension of 
the spherical harmonics for $S^{m-1}$, and $\bar{h}_{k}$ be that for the real projective space.
By Corollary \ref{corollary} for the real projective space, 
\[
\frac{|X|}{2} \leq \sum_{k=0}^{\frac{s-\delta_s}{2}} \bar{h}_k -\bar{h}_{\frac{s}{2}+\frac{\delta_s}{2}-i}
\]
for $1+\delta_s \leq i \leq s+\delta_{s}$.  Note that $\bar{h}_k=h_{2k}$ \cite{Levenshtein}. 
Therefore, this corollary follows. 
\end{proof}

\section{Bounds on $Q$-polynomial schemes}
\begin{theorem}\label{S0}
Let $(X,\mathcal{R})$ be an $s$-class $Q$-polynomial scheme with respect to the ordering $E_0,E_1,\ldots,E_s$.
Define $l=\max\{k\in\{0,1,\ldots,s\}\mid a_0^*=\cdots=a_k^*=0\}$.
\begin{enumerate}
\item If $l=s$ holds, then $|X|\leq 2\binom{m+s-2}{s-1}$.
The equality holds if and only if $X$ is a $(2s-1)$-design and $m_i=h_{i}$ holds for $0\leq i\leq s-1$ and  $m_s=\binom{m+s-2}{s-1}-\binom{m+s-3}{s-2}$.
\item If $(s-1)/2\leq l\leq s-1$ holds, then  $|X|\leq \binom{m+2l-s}{2l+1-s}+\binom{m+s-1}{s}$ holds.
The equality holds if and only if 
\begin{align*}
m_i=
\begin{cases}h_{i,m} & \text{if}\ 2\leq i\leq l+1,\\
\sum_{k=0}^{i-l-2}(h_{i-2k,m}-h_{i-2k-1,m}) & \text{if}\ l+2\leq i\leq s.
\end{cases}
\end{align*}
\item If $l\leq (s-2)/2$ holds, then  $|X|\leq \binom{m+s-1}{s}$ holds.
The equality holds if and only if 
\begin{align*}
m_i=
\begin{cases}h_{i,m} & \text{if}\ 2\leq i\leq l+1,\\
\sum_{k=0}^{i-l-2}(h_{i-2k,m}-h_{i-2k-1,m}) & \text{if}\ l+2\leq i\leq  2l+2, \\
\binom{m+i-1}{i}-\binom{m+i-2}{i-1} & \text{if}\ 2l+3\leq i\leq s.
\end{cases}
\end{align*}
\end{enumerate}
Moreover, when the equality holds in each case $(2)$ or $(3)$, $X$ is a spherical $(2l+2)$-design. 
\end{theorem}
\begin{proof}
Suppose $l=s$, namely the scheme $(X,\mathcal{R})$ is $Q$-bipartite.
Then, by \cite[Corolalry~4.2]{MMW}, the image of the embedding of the scheme into the unit sphere is an antipodal set.
Therefore it follows from \cite{Delsarte-Goethals-Seidel2} that $|X|\leq 2\binom{m+s-2}{s-1}$.
When the equality holds, \cite[Theorem~6.8, Remark~7.6]{Delsarte-Goethals-Seidel} says that $m_i=Q_i(1)=h_{i}$ holds for $0\leq i\leq s-1$ and by (\ref{S5}), $m_s=\binom{m+s-2}{s-1}-\binom{m+s-3}{s-2}$.

Suppose $l\leq s-1$.
The three term recurrence and the conditions $a_1^*=\cdots=a_l^*=0$ implies that for $0\leq i\leq l+1$, there exist positive real numbers 
$f_{i,k}$ such that 
\begin{align}E_1^{\circ i}=\underset{k\equiv i \pmod 2}{\sum_{k=0}^i}f_{i,k}E_k.\label{S11}\end{align}
When $l+1\leq s-1$, the three term recurrence and $a_{l+1}^*\neq0$ imply that for $1\leq i\leq s-l-1$, there exist positive real numbers $f_{l+1+i,k}$  such that  
\begin{align}E_1^{\circ (l+1+i)}=\underset{k\equiv l+1+i\pmod2}{\sum_{k=0}^{l+1+i}}f_{l+1+i,k}E_k+\underset{k\equiv l+i\pmod2}{\sum_{k=\max\{l+2-i,0\}}^{ l+i}}f_{l+1+i,k}E_k.\label{S12}\end{align}

Taking the rank of the both hand sides in (\ref{S11}) and (\ref{S12}), it follows from Lemma~\ref{S3} that
\begin{align}
&\underset{k\equiv i\pmod2}{\sum_{k=0}^i}m_k\leq\binom{m+i-1}{i} \text{ for }0\leq i\leq l+1,\label{S1} \\
&\underset{k\equiv l+1+i\pmod2}{\sum_{k=0}^{l+1+i}}m_k+\underset{k\equiv l+i\pmod2}{\sum_{k=\max\{l+2-i,0\}}^{l+i}}m_k\leq\binom{m+l+i}{l+i+1}\text{ for }1\leq i\leq s-l-1\label{S2}.
\end{align} 

(2): Substituting $i=2l+1-s$ and $i=s-l-1$ into (\ref{S1}) and (\ref{S2}) respectively, 
using the equation (\ref{S5}), 
\begin{align*}
|X|&=\sum_{k=0}^sm_k \\
&\leq \underset{k\equiv 2l+1-s\pmod2}{\sum_{k=0}^{2l+1-s}}m_k+\underset{k\equiv s\pmod2}{\sum_{k=0}^s}m_k+\underset{k\equiv s-1\mod2}{\sum_{k=2l+3-s}^{s-1}}m_k \\
&\leq \binom{m+2l-s}{2l+1-s}+\binom{m+s-1}{s}.
\end{align*} 
The equalities hold in case (2) if and only if $\rank E_1^{\circ (2l-s)}=\binom{m+2l-s}{2l+1-s}$ and $\rank E_1^{\circ s}=\binom{m+s-1}{s}$.
By Lemma~\ref{S3} (2), this condition is equivalent to $\rank E_1^{\circ i}=\binom{m+i-1}{i}$ for $0\leq i\leq s$.
Then we obtain a system of linear equations from (\ref{S1}) for $2\leq i\leq l+1$ and (\ref{S2}) for $1\leq i\leq s-l-1$ whose unknowns are $\{m_i\mid 2\leq i\leq s\}$.
Its coefficient matrix is a lower triangluar matrix with non-zero diagonals.
Therefore the equality holds if and only if $m_i$ for $2\leq i\leq s$ are uniquely determined as desired. 

(3): Substituting $i=s-l-1$ into (\ref{S2}), using the equation (\ref{S5}), 
\begin{align*}
|X|&=\sum_{k=0}^s m_k \\
&=\underset{k\equiv s\pmod2}{\sum_{k=0}^s}m_k+\underset{k\equiv s-1\pmod2}{\sum_{k=0}^{s-1}}m_k \\
&\leq\binom{m+s-1}{s}.
\end{align*}
Since the same method in (2) is applied in (3), the equality holds if and only if $m_i$ for $2\leq i\leq s$ are uniquely determined as desired.

When the equality holds in each case (2), (3), $m_i=h_{i,m}$ holds for $0\leq i\leq l+1$.
Repeated application of the formula (\ref{S5}) together with (\ref{S4}), we have $c_j^*=mj/(m+2j-2)$ for $0\leq j\leq l+1$.
Recall we assume $a_i^*=0$ for $0\leq i\leq l$.
It follows from Theorem \ref{SHO} that $X$ is a spherical $(2l+2)$-design.
\end{proof}

\begin{remark} 
In Theorem~\ref{S0} when the equality holds for $l=s-1$ (resp.\ $l=s$), then $X$ is a tight $2s$-design (resp.\ tight $(2s-1)$-design) in $S^{m-1}$ \cite{Delsarte-Goethals-Seidel}. 
\end{remark}

\section{Examples}
\begin{example}
Let $\Omega$ be the minimum vectors, which rescaled to the norm $1$, of the Leech lattice in $\mathbb{R}^{24}$. Fix $u,v \in \Omega$ such that $(u,v)=-1/4$.
Define $X$ by $\{x\in \Omega \mid (x,u)=1/2,(x,v)=0\}$.
Then $|X|=2025$.
Considering the projection onto $\mathbb{R}^{22}$, we may regard $X$ as a subset in the $S^{21}$.
Then $X$ is a spherical $3$-distance set with strength $4$ in $S^{21}$.
Since $h_{0}=1$, $h_{1}=22$, and $h_{3}=2002$, $X$ attains the upper bound in Corollary~\ref{corollary}. 

On the other hand, since $X$ satisfies $t=2s-2$, $X$ carries a $Q$-polynomial scheme whose Krein matrix $B_1^*$ is
$$B_1^*=\begin{pmatrix}
0  & 1  & 0 & 0 \\
22 & 0  & 11/6 & 0 \\
0  & 21 & 27/22 & 30/11 \\
0  & 0  & 625/33 & 212/11  
\end{pmatrix}
.$$
Then the scheme $X$ also attains the bound in Theorem~\ref{S0}. 
\end{example}

\begin{example}
Let $X$ be the set of vertices of the dodecahedron in $\mathbb{R}^3$.
Then, $X$ is an antipodal spherical $5$-distance set with strength $5$.
Since $h_{0}=1$ and $h_{4}=9$, $X$ attains the upper bound in Corollary \ref{anti_coro}.

\end{example}
	

{\it Hiroshi Nozaki}\\
	Graduate School of Information Sciences, \\
	Tohoku University \\
	Aramaki-Aza-Aoba 09, \\
	Aoba-ku, \\
	Sendai 980-8579, \\
	Japan\\ 
	nozaki@ims.is.tohoku.ac.jp\\
	\quad\\
{\it Sho Suda}\\ 
	Graduate School of Information Sciences, \\
	Tohoku University \\
	Aramaki-Aza-Aoba 09, \\
	Aoba-ku, \\
	Sendai 980-8579, \\
	Japan\\ 
	suda@ims.is.tohoku.ac.jp\\

\end{document}